\documentclass[12pt]{article}

\usepackage{amsmath, epsfig, cite}
\usepackage{amssymb}
\usepackage{amsfonts}
\usepackage{latexsym}
\usepackage{amsthm}

\makeatletter
\renewcommand{\@seccntformat}[1]{{\csname the#1\endcsname}.\hspace{.5em}}
\makeatother

\newtheorem{thm}{Theorem}[section]
\newtheorem{prop}[thm]{Proposition}

\newtheorem{cor}[thm]{Corollary}
\newtheorem{conj}[thm]{Conjecture}
\newtheorem{lem}[thm]{Lemma}

\numberwithin{equation}{section}

\begin{document}

\renewcommand{\thefootnote}{*}

\begin{center}
{\Large\bf Some $q$-congruences with parameters}
\end{center}

\vskip 2mm \centerline{Victor J. W. Guo  }
\begin{center}
{\footnotesize School of Mathematical Sciences, Huaiyin Normal University, Huai'an 223300, Jiangsu,\\
 People's Republic of China\\
{\tt jwguo@hytc.edu.cn } }
\end{center}

\vskip 0.7cm \noindent{\small{\bf Abstract.}} Let $\Phi_n(q)$ be the $n$-th cyclotomic polynomial in $q$. Recently, the author and Zudilin devised a method, called
`creative microscoping', to prove some $q$-supercongruences mainly modulo $\Phi_n(q)^3$ by introducing an additional parameter $a$.
In this paper, we use this method to confirm some conjectures on $q$-supercongruences modulo $\Phi_n(q)^2$.
We also give some parameter-generalizations of known $q$-supercongruences. For instance, we present
further generalizations of a $q$-analogue of a famous supercongruence of Rodriguez-Villegas:
$$
\sum_{k=0}^{p-1}\frac{{2k\choose k}^2}{16^k} \equiv (-1)^{(p-1)/2}\pmod{p^2}\quad\text{for any odd prime $p$.}
$$



\section{Introduction}
It is well known that Ramanujan's and Ramanujan-type
formulae \cite{Ramanujan} for $1/\pi$, such as
\begin{equation}
\sum_{n=0}^\infty\frac{\binom{4n}{2n}{\binom{2n}{n}}^2}{2^{8n}3^{2n}}\,(8n+1)
=\frac{2\sqrt{3}}{\pi},
\label{ram1}
\end{equation}
may lead to Ramanujan-type supercongruences (see Van Hamme \cite{Hamme} and Zudilin \cite{Zudilin}).
In the example \eqref{ram1}, the result reads
\begin{equation}
\sum_{k=0}^{p-1}\frac{\binom{4k}{2k}{\binom{2k}{k}}^2}{2^{8k}3^{2k}}\,(8k+1)
\equiv p\left(\frac{-3}p\right)\pmod{p^3}
\quad\text{for $p>3$ prime},
\label{ram1a}
\end{equation}
where $\bigl(\frac{-3}{\cdot}\bigr)$ denotes the Jacobi--Kronecker symbol.

Recently, the author and Zudilin \cite{GuoZu} proved a $q$-analogue of \eqref{ram1a} by using a method referred to as {\it creative microscoping}. Specifically, to prove the following
$q$-analogue of \eqref{ram1a}:
\begin{align}
\sum_{k=0}^{n-1}\frac{(q;q^2)_k^2 (q;q^2)_{2k}}{(q^2;q^2)_{2k}(q^6;q^6)_k^2}[8k+1]q^{2k^2}
&\equiv q^{-(n-1)/2}[n]\left(\frac{-3}{n}\right) \pmod{\Phi_n(q)^3},
\label{q4a}
\end{align}
where $n$ is coprime with $6$,
we first establish the following $q$-congruence with an extra parameter $a$:
\begin{align}
&\sum_{k=0}^{n-1}\frac{(aq;q^2)_k (q/a;q^2)_k (q;q^2)_{2k}}{(q^2;q^2)_{2k}(aq^6;q^6)_k (q^6/a;q^6)_k }[8k+1]q^{2k^2} \notag\\
&\quad\equiv q^{-(n-1)/2}[n]\left(\frac{-3}{n}\right) \pmod{\Phi_n(q)(1-aq^n)(a-q^n)}.
\label{q4a-new}
\end{align}
Here and throughout the paper, the {\it $q$-shifted factorial} is defined by $(a;q)_n=(1-a)(1-aq)\cdots (1-aq^{n-1})$ for $n\geqslant 1$ and $(a;q)_0=1$,
while the {\it $q$-integer} is defined as $[n]=[n]_q=1+q+\cdots+q^{n-1}$. Since the polynomials $[n]$, $1-aq^n$, and $a-q^n$ are relatively prime,
the $q$-congruence \eqref{q4a-new} can be established modulo these three polynomials individually. It is easy to see that we recover \eqref{q4a} from \eqref{q4a-new}
by taking the limit $a\to 1$. It is worth mentioning that the creative microscoping method has caught the interests of Guillera \cite{Guillera3} and Straub \cite{Straub}.

In this paper, we shall prove some $q$-congruences modulo $\Phi_n(q)^2$ by the creative microscoping method. Some of our results confirm the corresponding conjectures in \cite{GPZ,GZ14,GZ15,GuoZu}.

For any rational number $x$ and positive integer $m$ such that the denominator of $x$ is relatively prime to $m$, we let $\langle x\rangle_m$ denote the {\it least non-negative residue} of $x$ modulo $m$.
Recall that the $n$-th {\it cyclotomic polynomial} $\Phi_n(q)$ is defined as
\begin{align*}
\Phi_n(q):=\prod_{\substack{1\leqslant k\leqslant n\\ \gcd(n,k)=1}}(q-\zeta^k),
\end{align*}
where $\zeta$ is an $n$-th primitive root of unity.
Our first result can be stated as follows.
\begin{thm}\label{thm:dnr-rs}
Let $d$, $n$ and $r$ be positive integers with $\gcd(d,n)=1$ and $n$ odd.
Let $s\leqslant n-1$ be a nonnegative integer with $s\equiv \langle -r/d\rangle_n+1\pmod 2$. Then
\begin{align}
\sum_{k=s}^{n-1}\frac{(aq^r;q^d)_k (bq^{d-r};q^d)_k (q^d;q^{2d})_k q^{dk}}{(q^d;q^d)_{k-s}(q^d;q^d)_{k+s} (abq^{2d};q^{2d})_k}
\equiv 0 \pmod{(1-aq^{r+d\langle -r/d\rangle_n})(1-bq^{d-r+d\langle (r-d)/d\rangle_n})}.  \label{eq:eq:conj6.1-ab-gen}
\end{align}
\end{thm}

It is easy to see that both $r+d\langle-r/d\rangle_n$ and $d-r+d\langle (r-d)/d\rangle_n$ are divisible by $n$,
and so the limit of $(1-aq^{r+d\langle-r/d\rangle_n})(a-q^{d-r+d\langle (r-d)/d\rangle_n})$ as $a,b\to1$ has the factor $\Phi_n(q)^2$. On the other hand,
all the denominators on the left-hand side of  \eqref{eq:eq:conj6.1-ab-gen} as $a,b\to1$ are relatively prime to $\Phi_n(q)$ since $\gcd(d,n)=1$ and $n$ is odd.
Thus, letting $a,b\to 1$ in \eqref{eq:eq:conj6.1-ab-gen}, we are led to the following result.
\begin{cor}\label{thm:conj6.1}
Let $d$, $n$ and $r$ be positive integers with $\gcd(d,n)=1$ and $n$ odd.
Let $s\leqslant n-1$ be a nonnegative integer with $s\equiv \langle -r/d\rangle_n+1\pmod 2$. Then
\begin{align*}
\sum_{k=s}^{n-1}\frac{(q^r;q^d)_k (q^{d-r};q^d)_{k} (q^d;q^{2d})_{k} q^{dk} }{ (q^d;q^d)_{k-s}(q^d;q^d)_{k+s} (q^{2d};q^{2d})_k }
\equiv 0 \pmod{\Phi_n(q)^2}.
\end{align*}
\end{cor}

We point out that, when $n$ is an odd prime, the above corollary confirms a conjecture of Guo and Zeng \cite[Conjecture~6.1]{GZ15}. It should also be mentioned that
the $d=2$, $r=1$, and $n$ being an odd prime $p$ case of Theorem~\ref{thm:conj6.1} has already been
given by Guo and Zeng \cite[Theorem 1.1]{GZ15}, and another special case where $n=p$ and $s\leqslant\min\{\langle -r/d\rangle_p,\langle -(d-r)/d\rangle_p\}$
has also been  given by Guo and Zeng \cite[Theorem~1.3]{GZ15}.
But the proofs given there are rather complicated (10 pages in total). Moreover, when $s=0$ and $q\to 1$, Corollary~\ref{thm:conj6.1} reduces to a result of Sun \cite[Theorem 2.5]{SunZH2}.

Another main result of this paper is as follows.
\begin{thm}\label{thm:abx}
Let $d$, $n$ and $r$ be positive integers with $\gcd(d,n)=1$ and $n$
odd. Then, modulo $(1-aq^{r+d\langle
-r/d\rangle_n})(1-bq^{d-r+d\langle (r-d)/d\rangle_n})$,
\begin{align}
\sum_{k=0}^{n-1}\frac{(aq^r;q^d)_k (bq^{d-r};q^d)_k (x;q^{d})_k q^{dk}}{(q^d;q^d)_k (abq^{2d};q^{2d})_k}
\equiv (-1)^{\langle -r/d\rangle_n}\sum_{k=0}^{n-1}\frac{(aq^r;q^d)_k (bq^{d-r};q^d)_k
(-x;q^{d})_k q^{dk}}{(q^d;q^d)_k (abq^{2d};q^{2d})_k}. \label{eq:ab-ab}
\end{align}
\end{thm}

It is clear that when $a,b\to 1$ and $n$ is an odd prime, the above theorem confirms a conjecture of Guo and Zeng \cite[Conjecture 7.3]{GZ14}.
Moreover, letting $b=1/a$ and $x=-1$ in \eqref{eq:ab-ab} and noticing that
$$
\frac{(-1;q^d)_k}{(q^{2d};q^{2d})_k}
=\frac{2}{(q^d;q^d)_k (1+q^{dk})},
$$
we get the following conclusion.
\begin{cor}\label{conj:3.3}
Let $d$, $n$ and $r$ be positive integers with $\gcd(d,n)=1$ and $n$
odd. Then, modulo $(1-aq^{r+d\langle-r/d\rangle_n})(a-q^{d-r+d\langle (r-d)/d\rangle_n})$,
\begin{equation}
\sum_{k=0}^{n-1}\frac{2(aq^r;q^d)_k (q^{d-r}/a;q^d)_k
q^{dk}}{(q^d;q^d)_{k}(q^d;q^d)_{k} (1+q^{dk})} \equiv (-1)^{\langle-r/d\rangle_n}.  \label{eq:conj3.3-a}
\end{equation}
\end{cor}

Similarly as before, letting $a\to 1$ in \eqref{eq:conj3.3-a}, we get
\begin{equation*}
\sum_{k=0}^{n-1}\frac{2(q^r;q^d)_k (q^{d-r};q^d)_k q^{dk}}{(q^d;q^d)_{k}(q^d;q^d)_{k} (1+q^{dk})} \equiv (-1)^{\langle
-r/d\rangle_n} \pmod{\Phi_n(q)^2},
\end{equation*}
which was originally conjectured by Guo, Pan, and Zhang \cite[Conjecture 3.3]{GPZ}.

Using Theorem \ref{thm:abx}, we shall also confirm two conjectures of Guo and Zudilin \cite[Conjectures 4.15 and 4.16]{GuoZu}.

\begin{cor}\label{cor-1}
{\rm\cite[Conjecture 4.16]{GuoZu}}
Let $d$, $n$ and $r$ be positive integers with $r<d$, $\gcd(d,n)=1$, and $n$ odd.
Then, modulo $(1-aq^{n\langle r/n\rangle_d})(1-bq^{n\langle(d-r)/n\rangle_d})$,
\begin{align*}
\sum_{k=0}^{n-1}\frac{(aq^r;q^d)_k (bq^{d-r};q^d)_k (x;q^{d})_k q^{dk}}{(q^d;q^d)_k (abq^{2d};q^{2d})_k}
\equiv (-1)^{\langle -r/d\rangle_n}\sum_{k=0}^{n-1}\frac{(aq^r;q^d)_k (bq^{d-r};q^d)_k (-x;q^{d})_k q^{dk}}{(q^d;q^d)_k (abq^{2d};q^{2d})_k}.
\end{align*}
\end{cor}

\begin{cor}\label{cor-2}
{\rm\cite[Conjecture 4.15]{GuoZu}}
Let $n$ be a positive odd integer. Then
\begin{align*}
&
\sum_{k=0}^{n-1}\frac{(aq;q^2)_k (bq;q^2)_k (x;q^2)_k q^{2k}}{(q^2;q^2)_k (abq^4;q^4)_k}
\\ &\qquad
\equiv (-1)^{(n-1)/2}\sum_{k=0}^{n-1}\frac{(aq;q^2)_k (bq;q^2)_k (-x;q^2)_k q^{2k}}
{(q^2;q^2)_k (abq^4;q^4)_k} \pmod{(1-aq^n)(1-bq^n)}.
\end{align*}
\end{cor}

Note that Corollary \ref{cor-2} is the $d=2$ and $r=1$ case of Corollary \ref{cor-1}.
The other results of this paper are closely related to a famous supercongruence of Rodriguez-Villegas (see \cite{Mortenson1,RV}):
\begin{align}
\sum_{k=0}^{p-1}\frac{{2k\choose k}^2}{16^k} \equiv (-1)^{(p-1)/2}\pmod{p^2}\quad\text{for any odd prime $p$.}  \label{eq:RV}
\end{align}
The author and Zeng \cite{GZ14} proved a $q$-analogue of \eqref{eq:RV}:
\begin{align}
\sum_{k=0}^{p-1}\frac{(q;q^2)_k^2}{(q^2;q^2)_k^2} \equiv (-1)^{(p-1)/2}q^{(1-p^2)/4}\pmod{[p]^2}\quad\text{for odd prime $p$}, \label{eq:GZ-RV}
\end{align}
which has been generalized by Guo, Pan, and Zhang \cite{GPZ} and Ni and Pan \cite{NP}.
In this paper we shall give some new parameter-generalizations of \eqref{eq:GZ-RV}, such as
\begin{thm}\label{thm:2.2}Let $n$ be a positive odd integer. Then
\begin{align}
\sum_{k=0}^{n-1}\frac{(aq;q^2)_k (q/a;q^2)_k}{(q^2;q^2)_k^2}
&\equiv (-1)^{(n-1)/2}q^{(1-n^2)/4}\pmod{(1-aq^n)(a-q^n)}. \label{eq:q-RV1}
\end{align}
\end{thm}

The paper is organized as follows. We shall prove Theorems~\ref{thm:dnr-rs} and \ref{thm:abx} in Sections~2 and 3, respectively.
In Section~4, we shall give a more general form of Theorem~\ref{thm:2.2} (see Theorem~\ref{thm:q-Tauraso}) along with some other parameter-generalizations of \eqref{eq:GZ-RV}.

\section{Proof of Theorem \ref{thm:dnr-rs}}
For $a=q^{-r-d\langle -r/d\rangle_n}$, the left-hand side of
\eqref{eq:eq:conj6.1-ab-gen} is equal to
\begin{align}
&\sum_{k=s}^{n-1}\frac{(q^{-r_1d};q^d)_k (bq^{d-r};q^d)_k (q;q^{2d})_k q^{dk}}{(q^d;q^d)_{k-s}(q^d;q^d)_{k+s} (bq^{2d-r_1d-r};q^{2d})_k} \notag\\[5pt]
&\quad=\frac{(q^{-r_1d};q^d)_s (bq^{d-r};q^d)_s (q;q^{2d})_s q^{dk}}{(q^d;q^d)_{2s} (bq^{2d-r_1d-r};q^{2d})_s}  \notag\\
&\qquad\times{}_{4}\phi_{3}\left[\!\!\begin{array}{c}
q^{(s-r_1)d},\, bq^{(s+1)d-r},\,q^{(s+1/2)d},\, -q^{(s+1/2)d}\\[5pt]
\sqrt{b}q^{(s+1)d-(r_1d+r)/2},\, -\sqrt{b}q^{(s+1)d-(r_1d+r)/2},\, q^{(2s+1)d}\end{array}\!\!;q^d,q^d\right], \label{eq:a-fixed}
\end{align}
where $r_1=\langle -r/d\rangle_n$ the {\it basic hypergeometric series $_{r+1}\phi_r$} is defined as
$$
_{r+1}\phi_{r}\left[\begin{array}{c}
a_1,a_2,\ldots,a_{r+1}\\
b_1,b_2,\ldots,b_{r}
\end{array};q,\, z
\right]
=\sum_{k=0}^{\infty}\frac{(a_1;q)_k(a_2;q)_k\cdots(a_{r+1};q)_k z^k}
{(q;q)_k(b_1;q)_k(b_2;q)_k\cdots(b_{r};q)_k}.
$$
If $s>r_1$, then $(q^{-r_1d};q^d)_s=0$ and so the right-hand side of \eqref{eq:a-fixed} is equal to $0$.
If $s\leqslant r_1$, then by the assumption, we have $s-r_1\equiv 1\pmod{2}$, and therefore by Andrews' terminating $q$-analogue of Watson's formula (see \cite{Andrews76} or \cite[(II.17)]{GR}):
\begin{align*} 
 {}_{4}\phi_{3}\left[\!\!\begin{array}{c}
q^{-n},\, a^{2}q^{n+1},\,b,\, -b\\
aq,\, -aq,\, b^{2}\end{array}\!\!;q,q\right]
=\begin{cases}
0,&\text{if $n$ is odd},\\[5pt]
\displaystyle\frac{b^{n}(q, a^{2}q^{2}/b^{2}; q^{2})_{n/2}}
{(a^{2}q^{2},\, b^{2}q; q^{2})_{n/2}},& \text{if $n$ is even,}
\end{cases}
\end{align*}
we conclude that the right-hand side of \eqref{eq:a-fixed} is still equal to $0$. This proves that
\begin{align*}
\sum_{k=s}^{n-1}\frac{(aq^r;q^d)_k (bq^{d-r};q^d)_k (q^d;q^{2d})_k q^{dk}}{(q^d;q^d)_{k-s}(q^d;q^d)_{k+s} (abq^{2d};q^{2d})_k}
\equiv 0 \pmod{(1-aq^{r+d\langle -r/d\rangle_n})}.
\end{align*}
Since $\langle -r/d\rangle_n\equiv\langle -(d-r)/d\rangle_n\pmod{2}$ for odd $n$, by symmetry we have
\begin{align*}
\sum_{k=s}^{n-1}\frac{(aq^r;q^d)_k (bq^{d-r};q^d)_k (q^d;q^{2d})_k q^{dk}}{(q^d;q^d)_{k-s}(q^d;q^d)_{k+s} (abq^{2d};q^{2d})_k}
\equiv 0 \pmod{(1-bq^{d-r+d\langle (r-d)/d\rangle_n})}.
\end{align*}
The proof of \eqref{eq:eq:conj6.1-ab-gen} then follows from the fact that the polynomials $1-aq^{r+d\langle -r/d\rangle_n}$ and $1-bq^{d-r+d\langle (r-d)/d\rangle_n}$ are relatively prime.

\section{Proof of Theorem \ref{thm:abx}}
We first establish the following lemma, which is a generalization of \cite[Lemma 5.2]{GZ14}.
\begin{lem}\label{lem:new-legendre}
Let $n$ be a positive integer and
\begin{align*}
F_n(x,b,q)=\sum_{k=0}^n \frac{(q^{-n};q)_k (b;q)_k (x;q)_k}{(q;q)_k (bq^{1-n};q^2)_k} q^k
\end{align*}
Then
\begin{align}
F_n(x,b,q)=(-1)^n F_n(-x,b,q). \label{eq:fnxq-sym}
\end{align}
\end{lem}
\begin{proof}[First Proof] Recall that the $q$-binomial theorem (see \cite[p.~36, Theorem 3.3]{Andrews}) can be stated as follows:
\begin{align*}
(x;q)_N=\sum_{k=0}^N {N\brack k}(-x)^k q^{k(k-1)/2},
\end{align*}
where the {\it $q$-binomial coefficients} ${n\brack k}$
are defined by
$$
{n\brack k}={n\brack k}_q
=\begin{cases}\displaystyle\frac{(q;q)_n}{(q;q)_k (q;q)_{n-k}} &\text{if $0\leqslant k\leqslant n$,} \\[5pt]
0 &\text{otherwise.}
\end{cases}
$$
Thus, the coefficient of $x^j$ ($0\leqslant j\leqslant n$)
in $F_n(x,b,q)$ is given by
\begin{align}
&(-1)^{j}q^{j(j-1)/2}\sum_{k=j}^{n} \frac{(q^{-n};q)_k (b;q)_k }{(q;q)_k (bq^{1-n};q^2)_k}{k\brack j} q^k \notag\\
&\quad=\frac{(-1)^{j} q^{j(j-1)/2}}{(q;q)_j}\sum_{k=j}^{n} \frac{(q^{-n};q)_k (b;q)_k }{(q;q)_{k-j} (bq^{1-n};q^2)_k} q^k  \notag \\
&\quad=\frac{(-1)^{j} q^{j(j+1)/2} (q^{-n};q)_j (b;q)_j}{(q;q)_j (bq^{1-n};q^2)_j}
\sum_{k=j}^{n} \frac{(q^{j-n};q)_{k-j} (bq^j;q)_{k-j} }{(q;q)_{k-j} (bq^{2j+1-n};q^2)_{k-j}} q^{k-j}.  \label{eq:sum-bxn}
\end{align}
Moreover, letting $a=q^{j-n}$ and $b\to bq^{j}$ in  Andrews' $q$-analogue of Gauss' $_2F_1(-1)$ sum
 (see \cite{Andrews73,Andrews74} or \cite[Appendix (II.11)]{GR}):
\begin{align*}
\sum_{k=0}^\infty \frac{(a;q)_k(b;q)_k q^{k(k+1)/2}}{(q;q)_k(abq;q^2)_k}
=\frac{(aq;q^2)_\infty (bq;q^2)_\infty}{(q;q^2)_\infty (abq;q^2)_\infty},
\end{align*}
we get
\begin{align}
\sum_{k=0}^{n-j} \frac{(q^{j-n};q)_{k} (bq^j;q)_{k}q^{k(k+1)/2}}
{(q;q)_{k} (bq^{2j+1-n};q^2)_{k}}
&=\frac{(q^{j+1-n};q^2)_\infty (bq^{j+1};q^2)_\infty}{(q;q^2)_\infty (bq^{2j+1-n};q^2)_\infty} \notag \\
&=\begin{cases}
\displaystyle \frac{(q^{j+1-n};q^2)_{(n-j)/2}}{(bq^{2j+1-n};q^2)_{(n-j)/2}}, &\text{if $j\equiv n\pmod 2,$}\\[10pt]
0, &\text{otherwise.}  \label{eq:sum-terminate}
\end{cases}
\end{align}
Finally, replacing $b$ and $q$ by $b^{-1}$ and $q^{-1}$ respectively in \eqref{eq:sum-terminate} and making some simplifications, we see that the summation on the right-hand side of \eqref{eq:sum-bxn}
is equal to $0$ for $j\not\equiv n\pmod{2}$.
 This proves \eqref{eq:fnxq-sym}.
\end{proof}

\begin{proof}[Second proof] (By Christian Kratthenthaler, personal communication.) It is easy to see that
\begin{align}
F_n(x,b,q)=
{}_{3}\phi_{2}\left[\!\!\begin{array}{c}
x,\,b,\,q^{-n},\\
\sqrt{b}q^{(1-n)/2},\, -\sqrt{b}q^{(1-n)/2}\end{array}\!\!;q,q\right].  \label{eq:3phi2}
\end{align}
Applying the transformation formula (see \cite[Appendix (III.12)]{GR})
\begin{align*}
{}_{3}\phi_{2}\left[\!\!\begin{array}{c}
b,\,c,\,q^{-n},\\
d,\, e\end{array}\!\!;q,q\right]
=c^n\frac{(e/c;q)_n}{(e;q)_n}
{}_{3}\phi_{2}\left[\!\!\begin{array}{c}
q^{-n},\, c,\, d/b\\
d,\, cq^{1-n}/e\end{array}\!\!;q,\frac{bq}{e}\right],
\end{align*}
we obtain
\begin{align*}
F_n(x,b,q)&=\frac{b^n (-q^{(1-n)/2}/\sqrt{b};q)_n}{(-\sqrt{b}q^{(1-n)/2};q)_n}
{}_{3}\phi_{2}\left[\!\!\begin{array}{c}
q^{-n},\, b,\, \sqrt{b}q^{(1-n)/2}/x\\
\sqrt{b}q^{(1-n)/2},\, -\sqrt{b}q^{(1-n)/2}\end{array}\!\!;q,-\frac{q^{(1+n)/2}x}{\sqrt{b}}\right]\\[5pt]
&=b^{n/2}
{}_{3}\phi_{2}\left[\!\!\begin{array}{c}
q^{-n},\, b,\, \sqrt{b}q^{(1-n)/2}/x\\
\sqrt{b}q^{(1-n)/2},\, -\sqrt{b}q^{(1-n)/2}\end{array}\!\!;q,-\frac{q^{(1+n)/2}x}{\sqrt{b}}\right].
\end{align*}

If we first interchange the bottom parameters in the basic hypergeometric series \eqref{eq:3phi2} and then
apply the above transformation formula, we get
\begin{align*}
F_n(x,b,q)&=\frac{b^n (q^{(1-n)/2}/\sqrt{b};q)_n}{(\sqrt{b}q^{(1-n)/2};q)_n}
{}_{3}\phi_{2}\left[\!\!\begin{array}{c}
q^{-n},\, b,\, -\sqrt{b}q^{(1-n)/2}/x\\
-\sqrt{b}q^{(1-n)/2},\, \sqrt{b}q^{(1-n)/2}\end{array}\!\!;q,\frac{q^{(1+n)/2}x}{\sqrt{b}}\right]\\[5pt]
&=(-1)^n b^{n/2}
{}_{3}\phi_{2}\left[\!\!\begin{array}{c}
q^{-n},\, b,\, -\sqrt{b}q^{(1-n)/2}/x\\
\sqrt{b}q^{(1-n)/2},\, -\sqrt{b}q^{(1-n)/2}\end{array}\!\!;q,\frac{q^{(1+n)/2}x}{\sqrt{b}}\right].
\end{align*}
The symmetry/antisymmetry under the replacement of $x$ by $-x$ is now obvious.
\end{proof}

\begin{proof}[Third proof] (By Xinrong Ma, personal communication.) We need the $q$-Sheppard identity
\begin{align*}
{}_{3}\phi_{2}\left[\!\!\begin{array}{c}
q^{-n},\, b,\, c\\
d,\,e\end{array}\!\!;q,q\right]
=c^n \frac{(d/c;q)_n (e/c;q)_n}{(d;q)_n (e;q)_n}
{}_{3}\phi_{2}\left[\!\!\begin{array}{c}
q^{-n},\, c,\, bcq^{1-n}/de\\
cq^{1-n}/d,\,cq^{1-n}/e\end{array}\!\!;q,q\right],
\end{align*}
which is a combination of the transformation formulas \cite[(III.12)]{GR} and \cite[(III.13)]{GR}. Applying this identity with $c=x$,
we can write $F_n(x,b,q)$ as
\begin{align*}
x^n\frac{(\sqrt{b}q^{(1-n)/2}/x;q)_n (-\sqrt{b}q^{(1-n)/2}/x;q)_n}{(\sqrt{b}q^{(1-n)/2};q)_n (-\sqrt{b}q^{(1-n)/2};q)_n}
{}_{3}\phi_{2}\left[\!\!\begin{array}{c}
q^{-n},\, x,\, -x\\
xq^{(1-n)/2}/\sqrt{b},\,-xq^{(1-n)/2}/\sqrt{b}\end{array}\!\!;q,q\right].
\end{align*}
This again proves \eqref{eq:fnxq-sym}.
\end{proof}

\noindent{\it Remark. } A direct proof of Corollary~\ref{conj:3.3} is as follows.
For $a=q^{-r-d\langle -r/d\rangle_n}$ or $a=q^{d-r+d\langle (r-d)/d\rangle_n}$, the left-hand side of \eqref{eq:conj3.3-a} is equal to
\begin{align*}
{}_3\phi_2\left[\begin{matrix}q^{-d\langle -r/d\rangle_n},\,q^{d-d\langle -r/d\rangle_n},\,-1\\q^d,\,-q^d\end{matrix}\,;\,q^d,\,q^d\right]
=(-1)^{\langle -r/d\rangle_n}
\end{align*}
by the $q$-Pfaff-Saalsch\"utz summation formula (see \cite[(II.12)]{GR}):
\begin{align*}
{}_3\phi_2\left[\begin{matrix}q^{-n},\,a,\,b\\c,abc^{-1}q^{1-n}\end{matrix}\,;\,q,\,q\right]
=\frac{(c/a;q)_n (c/b;q)_n}{(c;q)_n (c/ab;q)_n}.
\end{align*}

It is easy to see that Corollary \ref{cor-1} follows from Theorem \ref{thm:abx} and the following proposition.
\begin{prop}\label{prop:dnr}
Let $d$, $n$ and $r$ be positive integers with $r<d$ and $\gcd(d,n)=1$. Then $r+d\langle -r/d\rangle_n=n\langle r/n\rangle_d$.
\end{prop}
\begin{proof}
Suppose that $\langle r/n\rangle_d=a$. Namely, $r/n\equiv a\pmod{d}$ and $1\leqslant a<d$. Therefore, there exists an integer $b$ such that
$an-r=bd$. Since $r<d$, we have $an-r>-d$ and so $b\geqslant 0$. Moreover, we have $-r/d\equiv b\pmod{n}$ and  $b=(an-r)/d<n$. This proves that
$\langle -r/d\rangle_n=b$. That is, the desired identity holds.
\end{proof}

\section{More $q$-congruences with parameters}
We first give a generalization of Theorem \ref{thm:2.2} (corresponding to $x=1$), which is also a generalization of
a theorem of Guo and Zeng \cite[Theorem 2.1]{GZ14} (corresponding to $a=1$) and therefore a generalization of a result of Tauraso \cite{Tauraso}.

\begin{thm}\label{thm:q-Tauraso}
Let $n$ be a positive odd integer. Then, modulo $(1-aq^n)(a-q^n)$,
\begin{align}
\sum_{k=0}^{n-1}\frac{(aq;q^2)_k (q/a;q^2)_k}{(q^2;q^2)_k^2} x^k
\equiv
\sum_{k=0}^{(n-1)/2}{(n-1)/2\brack k}_{q^2}^2 q^{k^2-nk} (-x)^k (x;q^2)_{(n-1)/2-k}.
\label{eq:q-Tauraso}
\end{align}
\end{thm}
\begin{proof}Recall that the little $q$-Legendre polynomials are defined by
\begin{align*}
P_n(x|q)=\sum_{k=0}^{n}{n\brack k}{n+k\brack k} q^{k(k+1)/2-nk} (-x)^k,
\end{align*}
which can also be written as (see \cite{VAssche})
\begin{align*}
P_n(x|q)=(-1)^n q^{-\frac{n(n+1)}{2}}\sum_{k=0}^{n}{n\brack k}{n+k\brack k}(-1)^k q^{k(k+1)/2-nk} (xq;q)_k.
\end{align*}
Guo and Zeng \cite[Lemma 4.1]{GZ14} gave a new expansion for the little $q$-Legendre polynomials:
\begin{align}
P_n(x|q)=\sum_{k=0}^{n}{n\brack k}^2 q^{k(k+1)/2-nk} (-x)^k (xq;q)_{n-k}. \label{eq:pnxq-3}
\end{align}

For $a=q^{-n}$ or $a=q^n$, the left-hand side of \eqref{eq:q-Tauraso} is equal to
\begin{align*}
\sum_{k=0}^{(n-1)/2}\frac{(q^{1-n};q^2)_k (q^{n+1};q^2)_k}{(q^2;q^2)_k^2} x^k
=P_{(n-1)/2}(x|q^2),
\end{align*}
which is the right-hand side of \eqref{eq:q-Tauraso} by \eqref{eq:pnxq-3}. This proves the congruence \eqref{eq:q-Tauraso}.
\end{proof}

Similarly, by \eqref{eq:pnxq-3}, we obtain another generalization of \eqref{eq:q-RV1}.
It is also a parameter-generalization of another theorem of Guo and Zeng \cite[Theorem 2.3]{GZ14}.
\begin{thm}\label{thm:vq-Tauraso}
Let $n$ be a positive odd integer. Then, modulo $(1-aq^n)(a-q^n)$,
\begin{align}
\sum_{k=0}^{n-1}\frac{(aq;q^2)_k (q/a;q^2)_k}{(q^2;q^2)_k^2} x^k
&\equiv (-1)^{(n-1)/2}q^{(1-n^2)/4}
\sum_{k=0}^{n-1}\frac{(aq;q^2)_k (q/a;q^2)_k}{(q^2;q^2)_k^2} q^{2k}(x;q^2)_k. \label{eq:q2-RV1}
\end{align}
\end{thm}
It is clear that, when $x=1$, the congruence \eqref{eq:q2-RV1} reduces to \eqref{eq:q-RV1}. On the other
hand, setting $x=0$ in \eqref{eq:q2-RV1}, we obtain the following dual form of \eqref{eq:q-RV1}.
\begin{cor}Let $n$ be a positive odd integer.  Then
\begin{align*}
\sum_{k=0}^{n-1}\frac{(aq;q^2)_k (q/a;q^2)_k}{(q^2;q^2)_k^2}q^{2k}
&\equiv (-1)^{(n-1)/2}q^{(n^2-1)/4}\pmod{(1-aq^n)(a-q^n)}. 
\end{align*}
\end{cor}

The congruence \eqref{eq:q-RV1} has the following more general form, which is a parameter-generalization of a theorem of Guo and Zeng \cite[Theorem 1.6]{GZ15}.
\begin{thm}\label{thm:2k2k-r}
Let $n$ be a positive odd integer and let $0\leqslant s\leqslant (n-1)/2$. Then
\begin{align}
\sum_{k=0}^{(n-1)/2}\frac{(aq;q^2)_k (q/a;q^2)_{k+s}}{(q^2;q^2)_k (q^2;q^2)_{k+s} }
&\equiv (-1)^{(n-1)/2}q^{(1-n^2)/4}\pmod{(1-aq^n)(a-q^n)}. \label{eq:fin-gen-s}
\end{align}
\end{thm}
\begin{proof}For $a=q^{-n}$, the left-hand side of \eqref{eq:fin-gen-s} can be written as
\begin{align}
\sum_{k=0}^{(n-1)/2}\frac{(q^{1-n};q^2)_k (q^{n+1};q^2)_{k+s}}{(q^2;q^2)_k (q^2;q^2)_{k+s} }
=\sum_{k=0}^{(n-1)/2}(-1)^k {(n-1)/2\brack k}_{q^2}{(n-1)/2+k+s\brack (n-1)/2}q^{k^2-nk}.  \label{eq:sum-nks}
\end{align}
By the easily proved identity \cite[(5.1)]{GZ15}:
\begin{align*}
\sum_{k=0}^N (-1)^k {N\brack k} {M+k\brack N} q^{{k\choose 2}-Nk}
=(-1)^N q^{-{N+1\choose 2}}\quad\text{for $0\leqslant s\leqslant N$},
\end{align*}
we see that the right-hand side of \eqref{eq:sum-nks} is equal to $(-1)^{(n-1)/2}q^{(1-n^2)/4}$. This proves that the congruence
\eqref{eq:fin-gen-s} holds modulo $(1-aq^n)$. Similarly, we can show that it also holds modulo $a-q^n$.
\end{proof}

Let us now turn to another $q$-congruence story. In 2010, Sun and Tauraso \cite[Corollary 1.1]{ST2} proved that
\begin{align}
\sum_{k=0}^{p^{r}-1}\frac{1}{2^k}{2k\choose k}\equiv (-1)^{(p^r-1)/2}\pmod{p}, \label{eq:sun-tauraso}
\end{align}
where $p$ is an odd prime and $r$ is a positive integer. In the same year, Sun \cite{Sun2010} further showed that the above congruence holds modulo $p^2$.
Still in 2010, the author and Zeng \cite[Corollary 4.2]{GZ10} gave the following $q$-analogue of \eqref{eq:sun-tauraso}:
\begin{align}
\sum_{k=0}^{n-1}\frac{q^k}{(-q;q)_{k}}{2k\brack k}\equiv (-1)^{(n-1)/2}q^{(n^2-1)/4} \pmod{\Phi_n(q)}, \label{eq:guo-zeng-1}
\end{align}
where $n$ is a positive odd integer. In 2013, Tauraso \cite{Tauraso2} proved some $q$-supercongruences and conjectures that \eqref{eq:guo-zeng-1} is also true modulo $\Phi_n(q)^2$ for
any odd prime $n$.

It is natural that Tauraso's conjecture can be generalized as follows.
\begin{conj}\label{thm:Tauraso-1}
Let $n$ be a positive odd integer. Then
\begin{align}
\sum_{k=0}^{n-1}\frac{q^k}{(-q;q)_{k}}{2k\brack k}\equiv (-1)^{(n-1)/2}q^{(n^2-1)/4} \pmod{\Phi_n(q)^2}. \label{eq:Tauraso-1}
\end{align}
\end{conj}

It is easy to see that, for odd $n$, $\Phi_n(q^2)=\Phi_n(q)\Phi_n(-q)$, and therefore, \eqref{eq:Tauraso-1} is equivalent to
\begin{align}
\sum_{k=0}^{n-1}\frac{q^{2k}}{(-q^2;q^2)_{k}}{2k\brack k}_{q^2}\equiv (-1)^{(n-1)/2}q^{(n^2-1)/2} \pmod{\Phi_n(q)^2},  \label{eq:Tauraso-11}
\end{align}
by noticing that the left-hand side of \eqref{eq:Tauraso-11} remains unchanged when replacing $q$ by $-q$.
Moreover, we have
\begin{align*}
\frac{q^{2k}}{(-q^2;q^2)_{k}}{2k\brack k}_{q^2}
=\frac{(q^2;q^4)_k q^{2k}}{(q^2;q^2)_k}
=\frac{(q;q^2)_k (-q;q^2)_k q^{2k}}{(q^2;q^2)_k}.
\end{align*}
Conjecture \ref{thm:Tauraso-1} has recently been confirmed by the author \cite{GuoIJNT}.  Here we shall give the following parameter-generalization \eqref{eq:Tauraso-2}  of its weaker form \eqref{eq:guo-zeng-1}.
\begin{thm}\label{thm:Tauraso-2}
Let $n$ be a positive odd integer. Then
\begin{align}
\sum_{k=0}^{n-1}\frac{(aq;q^2)_k (-q/a;q^2)_k q^{2k}}{(q^2;q^2)_k} &\equiv (-1)^{(n-1)/2} q^{(n^2-1)/2} \pmod{(1-aq^n)(a+q^n)},\label{eq:Tauraso-2}  \\
\sum_{k=0}^{n-1}\frac{(aq;q^2)_k (q/a;q^2)_k q^{2k}}{(q^2;q^2)_k} &\equiv q^{(n^2-1)/2} \pmod{(1-aq^n)(a-q^n)}.\label{eq:Tauraso-22}
\end{align}
\end{thm}
\begin{proof} When $a=q^{-n}$ or $a=-q^{n}$, the left-hand side of \eqref{eq:Tauraso-2} is equal to
\begin{align*}
\sum_{k=0}^{(n-1)/2}\frac{(q^{1-n};q^2)_k (-q^{n+1};q^2)_k q^{2k}}{(q^2;q^2)_k}=(-1)^{(n-1)/2} q^{(n^2-1)/2}
\end{align*}
by the $q$-Chu-Vandermonde summation formula (see \cite[Appendix (II.6)]{GR}):
\begin{align*}
{}_3\phi_2\left[\begin{matrix}a,\,q^{-n}\\c\end{matrix}\,;\,q,\,q\right]=\frac{(c/a;q)_n }{(c;q)_n }a^n.
\end{align*}
This proves \eqref{eq:Tauraso-2}. Similarly, we can prove \eqref{eq:Tauraso-22}.
\end{proof}

Note that, the $a\to 1$ case of \eqref{eq:Tauraso-22} gives
$$
\sum_{k=0}^{n-1}\frac{(q;q^2)_k ^2 q^{2k}}{(q^2;q^2)_k} \equiv q^{(n^2-1)/2}\pmod{\Phi_n(q)^2}.
$$
But this congruence gives nothing when $n$ is an odd prime and  $q\to 1$.

\vskip 5mm \noindent{\bf Acknowledgments.}
The author would like to thank the anonymous referee for a careful reading and many helpful comments.
This work was partially
supported by the National Natural Science Foundation of China (grant 11771175).

\end{document}